\documentclass[a4paper]{amsart}
\usepackage{amsfonts}
\usepackage{amssymb}
\usepackage{graphicx}    
\usepackage{multicol}    

\usepackage{amsmath}
\usepackage[hidelinks]{hyperref}
\usepackage{tcolorbox}
\usepackage{mathrsfs}
\usepackage{amsthm}
\usepackage{lastpage}
\usepackage{fancyhdr}
\usepackage{accents}
\usepackage{stmaryrd}
\usepackage{tabularx}
\usepackage{tikz}
\usepackage{tikz-cd}
\usepackage{array}
\usepackage{mathtools}
\usepackage{comment}
\usepackage{subcaption}
\usepackage{adjustbox}
\usepackage{quiver}
\usepackage{colortbl}
\usepackage{pifont}
\usepackage{nicematrix}

\DeclareFontFamily{OT1}{pzc}{}
\DeclareFontShape{OT1}{pzc}{m}{it}{<-> s * [1.10] pzcmi7t}{}
\DeclareMathAlphabet{\mathpzc}{OT1}{pzc}{m}{it}
\newtheorem{thm}{Theorem}[section]
\newtheorem{prop}[thm]{Proposition}
\newtheorem{cor}[thm]{Corollary}
\newtheorem{lem}[thm]{Lemma}

\newtheorem*{thma}{Theorem \ref{Truemain}}
\newtheorem*{thmb}{Theorem \ref{poly}}
\theoremstyle{definition}
\newtheorem{defin}[thm]{Definition}
\newtheorem{construction}[thm]{Construction}
\newtheorem{exm}[thm]{Example}

\newtheorem{note}[thm]{Note}

\def\Z{\mathbb{Z}}    

\def\K{\mathcal{K}}
\def\H{\tilde{H}}
\def\A{\mathcal{A}}

\def\ZZ{\mathcal{Z}}

\def\Z{\mathbb{Z}}

\def\ker{\text{ker }}
\def\im{\text{Im }}

\def\F{\mathbb{F}}

\usepackage{etoolbox}
\makeatletter
\patchcmd{\@maketitle}
  {\ifx\@empty\@dedicatory}
  {\ifx\@empty\@date \else {\vskip3ex \centering\footnotesize\@date\par\vskip1ex}\fi
   \ifx\@empty\@dedicatory}
  {}{}
\patchcmd{\@adminfootnotes}
  {\ifx\@empty\@date\else \@footnotetext{\@setdate}\fi}
  {}{}{}
\makeatother

\newtheorem{innercustomgeneric}{\customgenericname}
\providecommand{\customgenericname}{}
\newcommand{\newcustomtheorem}[2]{%
  \newenvironment{#1}[1]
  {%
   \renewcommand\customgenericname{#2}%
   \renewcommand\theinnercustomgeneric{##1}%
   \innercustomgeneric
  }
  {\endinnercustomgeneric}
}\newcustomtheorem{customprop}{Proposition}
\newcustomtheorem{customthm}{\textbf{Theorem}}

\begin{document}

\title{Double, über and poset homology}

\author{Carlos Gabriel Valenzuela Ruiz}

\date{\today}

\maketitle
\thispagestyle{empty}
\begin{abstract}
   We present a comparison map between the überhomology of a simplicial complex and the double homology of its associated moment-angle complex. We show these two homology theories differ at three bidegrees, which depend on whether the complex $\K$ is neighbourly or not.
\end{abstract}

\maketitle
\setcounter{tocdepth}{1}
\tableofcontents
\section{Introduction}
   In \cite{uber}, the authors discussed a comparison between two new homological constructions for finite simplicial complexes (see Definition \ref{simplicialcomplex}). In the article, they showed that these two constructions are isomorphic almost everywhere, and further, they suggested that this connection goes significantly deeper as both theories appear as a second page of some spectral sequence.\\ 
   
   \indent The first construction is the \textit{double homology of a moment-angle complex}. Given a simplicial complex $\K$ on $[m]$, we can construct its associated \textit{moment-angle complex} $\ZZ_\K$, which is a subspace of the polydisk ${\left(D^2\right)}^m$ that encodes the combinatorial structure of $\K$ (see \cite[Ch. 4]{ToricTop}). In \cite{Limonchenko_2023}, the authors defined a cochain complex structure on $H_*(\ZZ_\K)$; taking cohomology gives us the bigraded double homology of the complex, $DH_{*,*}(\ZZ_\K)$ (originally denoted by $HH_{*,*}(\ZZ_\K)$). This construction was designed to solve a stability problem when using Tor complexes in topological data analysis (see \cite{Torapps} and \cite{stab}). The idea behind this approach is that it offers a global perspective on a persistent diagram as it observes the evolution of every component $\K_J$ of the complex $\K$ obtained by restricting to every subset $J$ of the vertex set (see Definition \ref{Subcomplex}).\\
   
   \indent The second construction was designed to categorify the total domination problem in graphs (see \cite{catuber}). The author of \cite{Celoria} designed a triply-graded cohomology theory for simplicial complexes $\K$ called \textit{überhomology} $\ddot{H}_{*,*}^*(\K)$. This homology theory is obtained by considering bicolourings of the vertices of $\K$, which induce different filtrations on $H(\K)$; these filtrations allow us to construct a family of complexes whose homology gives us überhomology. For our purposes, we mostly care about the uberhomology of the lowest degree: The $0$\textit{-degree überhomology} is the bigraded module $\ddot{B}_{*}^*(\K):=\ddot{H}_{0,*}^*(\K)$.\\
   
   It is no coincidence that these two homology theories are related, as both are special cases of a more general homology theory. Given a poset $P$, we can consider it as a category (see Definition \ref{posetcat}); thus, we can consider the category of functors $P\to \A$ to some abelian category $\A$. The author of \cite{chandler}, inspired by Khovanov cohomology, constructed a theory of \textit{poset cohomology} for a special case of such functors.\\
   
   Double and überhomology are special cases of poset cohomology in the sense that 
      \[DH_{-k,2l}(\ZZ_\K)\cong H^l(\H_{l-k-1}(\K_-))\]
   \[\ddot{B}_q^l(\ZZ_\K)\cong H^l(H_q(\K_-)),\]
    where $H_{q}(\K_-):2^{[m]}\to \text{Ab}$ assigns each subset $J\subseteq[m]$ to $H_q(\K_J)$, and inclusions $J\subseteq L$ to the induced map in homology of the inclusion $\K_J\subseteq \K_L$; the functor $\H_{*}(\K_-):2^{[m]}\to Ab$ is defined analogously. The first isomorphism follows from the definition of $DH$; the second one was observed by the authors in \cite{uber} after Definition 1.3. The authors of the aforementioned article showed that uberhomology and double homology coincide in most bidegrees. Our main result completes this comparison by describing a map between them: 

   \begin{thma}
       Let $\K$ be a simplicial complex on $[m]$. There are maps $\phi_{l,q}:H^l(\H_q(\K_-)) \to H^l(H_q(\K_-))$ for every $q,l\geq 0$ which is an isomorphism whenever $l>2$ or $q>0$, furthermore, we have the following exact sequence 
       \begin{equation}
           0\to H^1(H_0(\K_-))\xrightarrow[]{}\Z\to H^2(\H_0(\K_-))\xrightarrow{\phi_{2,0}} H^2(H_0(\K_-))\to0 \tag{\ding{72}}       \end{equation}
       where $\phi_{2,0}$ has a section.
   \end{thma}
    The comparison that Theorem \ref{Truemain} presents is even more complete than it suggests, and we can see that by specializing to neighbourly complexes (see definition \ref{nei}). We present this complete comparison with field coefficients using the bigraded poincaré polynomial of the functors (see definition \ref{poly}).
   \begin{thmb}
    Let $\K$ be a simplicial complex on $[m]$ and $\F$ a field, then 
    \[\left(P(\H_*(\K_-;\F))-P(H_*(\K_-;\F))\right)(x,y)=\left\{\begin{array}{cl}
         x^{-1}-y& \text{if $\K$ is neighbourly},  \\
         x^{-1}+y^2&  \text{otherwise.}
    \end{array}\right.\]
   \end{thmb}

\section{Notation and preliminaries}
\subsection{Simplicial complexes}
\begin{defin}\label{simplicialcomplex}
    A simplicial complex $\K$ on an ordered set $S$ is a non-empty collection of subsets of $S$ such that whenever $\tau\subseteq\sigma\in \K$ we have that $\tau\in\K$. An element of cardinality $p+1$ of $\K$ is called a $p$-dimensional face of $\K$. 
\end{defin}
 \noindent Throughout this work, we only consider finite simplicial complexes without ghost vertices, that is, $\K$ includes every singleton of its vertex set. We'll usually denote the vertex set by $[m]:=\{x\in \Z:1\leq x\leq m\}$ with the order of the integers. 

\begin{defin}\label{Subcomplex}
     Let $\K$ be a simplicial complex on $[m]$ and $J\subseteq[m]$. We define the \textit{full subcomplex of $\K$ on $J$} as the simplicial complex on $J$ given by $\K_J:=\{\sigma\in\K:\sigma\subseteq J\}$.
\end{defin}
Our main result relies on the notion of neighbourliness, which measures how big the minimal non-faces of a simplicial complex are:
\begin{defin}\label{nei}
     For $p\in \Z$, a simplicial complex $\K$ is \textit{$p$-neighbourly} if $\sigma\in \K$ whenever $|\sigma|=p+1$. A simplicial complex is said to be \textit{neighbourly} if it is $1$-neighbourly. Equivalently, $\K$ is neighbourly if its 1-skeleton is the complete graph.
\end{defin}

\begin{exm}
    An important class of simplicial complexes is the cycles; let $m\in \Z$ be at least $3$, then $\mathcal{C}^m$ is the simplicial complex on $[m]$ whose maximal faces are $\{a,b\}$ where $|b-a|\equiv 1\text{ mod }m$. The only neighbourly cycle is the $3$-cycle since $\{1,3\}\notin C^{m}$ for $m>3$.
\end{exm}

 \subsection{Poset categories and cohomology} 
 
 \begin{defin}\label{posetcat}
     Given a poset $(P,\leq)$, we can define the category  $\text{Cat}(P,\leq)$ whose objects are the elements of $P$ and there's a unique map $a\to b$ for $a,b\in P$ if and only if $a\leq b$. When context suffices, we shall abuse notation and denote $\text{Cat}(P,\leq)$ simply as $P$. In particular, we'll consider posets of the form $(2^S, \subseteq)$ for some set $S$, that is, the subsets of $S$ ordered by inclusion. 
 \end{defin}

\begin{construction}\label{posethom}
    Consider an abelian category $\A$ and an integer $m$. For every functor $F:2^{[m]}\to \A$ we can define a cochain complex $C^*(F)$ as follows:\\
    \begin{itemize}
        \item The objects in the complex are given by \[C^l(F)=\bigoplus\limits_{\substack{J\subseteq[m]\\|J|=l}} F(J)\]

        \item The differential is given by 
        \[d(F)=\sum_{J\subseteq[m]}\sum_{x\in[m]-J}(-1)^{\varepsilon(J;x)}F\left(J\xhookrightarrow{}J\cup\{x\}\right)\]
        where $\varepsilon(J;x)=|\{j\in J:j<x\}|$.
    \end{itemize}
\end{construction}
\noindent We define the cohomology of $F$ as the cohomology of this complex and denote it as $H^*(F)$. This construction is explored in further generality in chapter 6 of \cite{chandler}.\\

For an abelian category $\A$ consider the category Fun$(2^{[m]},\A)$ of functors $2^{[m]}\to \A$ and natural transformations. This category is abelian with structure inherited from $\A$.\newpage
\begin{prop}
    Construction 2.6 defines an exact functor $C^*(-):\text{Fun}(2^{[m]},\A)\to \text{dg} \A$ into the category of differential graded objects of $\A$.
\end{prop}
\begin{proof}
    Functoriality is shown in much more generality in Section 7 of \cite{chandler}. Let $E\xrightarrow{\eta}F\xrightarrow{\mu} G$ be exact in Fun$(2^{[m]},\A)$, we want to show that the sequence \[C(E)\xrightarrow{C(\eta)} C(F)\xrightarrow{C(\mu)}C(G)\] is exact, that is, that $\ker (C(\mu))\cong \im (C(\eta))$. Clearly $\im C(\eta)\subseteq \ker C(\mu)$ since by functoriality \[C(\mu)C(\eta)=C(\mu\eta)=C(0)=0.\]
    Let $\sigma\in \ker (C(\mu))$, then for every integer $l$ we have that $\mu_l(\sigma)=0$ and so $\sigma\in \ker \mu$. By assumption this means that $\sigma \in Im(\eta)$ and so $\sigma\in \im (C(\eta))$. As $\sigma$ was arbitrary it follows that $\ker (C(\mu))\subseteq \im (C(\eta))$ completing the proof.    
\end{proof}
\begin{exm}
    Let $\K$ be a simplicial complex and $F$ be the functor defined as follows:
    \begin{itemize}
        \item $F(\sigma)=\left\{\begin{array}{cl}
        \Z f_\sigma &  \text{if $\sigma\in\K$}\\
        0 & \text{else} 
    \end{array}\right.$
    where $f_\sigma$ is the indicator function of the face $\sigma$.
        \item Whenever $\sigma\in \K$ and $x\in \sigma$ \[F(\sigma\setminus\{x\}\subseteq\sigma)(f_\tau)=\left\{\begin{array}{cl}
            f_\sigma & \text{if }\tau=\sigma\setminus\{x\} \\
            0 & \text{else.}
        \end{array}\right.\]
    \end{itemize}
    this gives the reduced simplicial cochain complex, $C^l(F)\cong \tilde{C}^{l-1}(\K)$
\end{exm}

\begin{defin}
        For a simplicial complex $\K$ on $[m]$ and $q\in \Z$, we define the functor $H_{q}(\K_-):2^{[m]}\to \text{Ab}$ as one that maps $J\subseteq[m]$ to $H_q(\K_J)$ and maps the inclusion $J\subseteq L\subseteq[m]$ to the map induced by $\K_J\xhookrightarrow{}\K_L$ in homology. We define the functor $\H_q(\K_-):2^{[m]}\to \text{Ab}$ analogously.
\end{defin}
\begin{exm}
    Double (co)homology $DH_{*,*}(\ZZ_\K)$ is a bigraded functor on simplicial complexes designed in \cite{Limonchenko_2023} (denoted originally by $HH_{*,*}(\ZZ_\K)$) to solve a stability problem in topological data analysis. This is a special case of poset cohomology in the sense that
    \[DH_{-k,2l}(\ZZ_\K)\cong H^l(\H_{l-k-1}(\K_-))\]
    or equivalently
    \[DH_{q-l+1,2l}(\ZZ_\K)\cong H^l(\H_q(\K_-)).\]
    
\end{exm}
\begin{lem}\label{dh}
    $H^2(\H_0(\K_-))=0$ if and only if $\K$ is neighbourly.
\end{lem}
\begin{proof}
    If $\K$ is neighbourly $\K_{\{a,b\}}$ is always contractible and so $C^2(\H_0(\K_-))=0$. If $\K$ isn't neighbourly, there's a pair of disconnected vertices $\{a,b\}$. We know that $H^2(\H_0(\K_{\{a,b\}})=\Z$, and by Lemma $5.6$ in \cite{dohosphere} there's a surjective map $H^2(\H_0(\K))\to H^2(\H_0(\K_{\{a,b\}})$, which means that the domain can't be zero. 
\end{proof}

\begin{exm}
    Überhomology is a triply-graded functor on simplicial complexes designed in \cite{Celoria} to tackle the total dominating set problem in graph theory from a categorical perspective. As noted after Definition 1.3 in \cite{uber}, degree zero überhomology is a special case of poset cohomology as follows:
    \[\ddot{B}_q^l(\ZZ_\K)\cong H^l(H_{q}(\K_{-})).\]
\end{exm}
\begin{lem}\label{uber}
    $H^1(H_0(\K_-))=0$ if and only if $\K$ isn't neighbourly.
\end{lem}

\definecolor{bblue}{RGB}{0, 102, 204}
\begin{proof}
    Let $\K$ be neighbourly and let $J\subseteq[m]$ have $2$ elements, then $\K_J$ is contractible. Denote by $\sigma_J\in H_{0}(\K_J)$ the homology class of a vertex of $J$; if we denote by $d$ the differential of the poset $H_0(\K_-)$ we have that
\[d\left(\sum_{j\in[m]}\sigma_j\right)=\sum_{j\in[m]}d(\sigma_j)=\sum_{j\in[m]}\sum_{\substack{i\in[m]\\i\neq j}}(-1)^{\varepsilon(\{j\};i)}\sigma_{ij}=\sum_{\substack{i,j\in[m]\\i\neq j}}\pm(\sigma_{ij}-\sigma_{ji})=0.\]
Then $d^1$ has non-trivial kernel, and so $H^{1}(H_0(\K_-))\cong \ker d^1$ can't be zero.\\
    
    Now, let $\K$ be a non-neighbourly simplicial complex. Assume without loss of generality that $1$ isn't connected to every vertex. Further, assume that the set of vertices not-adjacent to $1$ is $[2,\lambda]$ for some $\lambda\geq 2$. Let $k\leq \lambda$, since $k$ is disconnected from $1$ we have that $H_0(\K_{\{1,k\}})\cong H_0(\K_{\{1\}})\oplus H_0(\K_{\{k\}})$.
    This identification (and the order inherited from the interval $[m]$) lets us choose a basis of $H_0(\K_{\{1,k\}})$ such that the restricted differential corresponds to the matrix
    $\begin{bNiceMatrix}
        \CodeBefore
        \begin{tikzpicture}
        \fill [blue!10,rounded corners] (1-|1) rectangle (3-|3);
        \end{tikzpicture}
        \Body
        -1 & 0\\
        0 & 1
    \end{bNiceMatrix}$. On the other hand, for $k>\lambda$, as it's connected to $1$, restricting the differential to the edge $\{1,k\}$ results in the matrix $\begin{bNiceMatrix}
        \CodeBefore
        \begin{tikzpicture}
        \fill [teal!10,rounded corners] (1-|1) rectangle (3-|3);
        \end{tikzpicture}
        \Body
        -1 & 1
    \end{bNiceMatrix}$. This tells us that the first $m+\lambda-2$ rows of a matrix pressentation of the differential of $H_0(\K_-)$ are as follows:
\[\phantom{}\;\;\;\;\;\;\;\;\;\;\;\;\:\;\;\;\;\;\;\begin{array}{cccccccc}
    1 & 2 & 3 & \cdots & \lambda &\lambda{+1} & \cdots & m
\end{array}\]
\[\begin{array}{c}
12\\13\\\vdots\\1\lambda\\1(\lambda{+1})\\\vdots\\1m\\\vdots\end{array}\begin{bNiceMatrix}[margin]  
\CodeBefore
  \begin{tikzpicture}
  \fill [blue!10,rounded corners] (1-|1) rectangle (5-|2) ;
  \fill [blue!10,rounded corners] (1-|2) rectangle (2-|3) ;
  \fill [blue!10,rounded corners] (2-|3) rectangle (3-|4) ;
  \fill [blue!10,rounded corners] (3-|4) rectangle (4-|5) ;
  \fill [blue!10,rounded corners] (4-|5) rectangle (5-|6) ;
  \fill [red!10,rounded corners]  (1-|6) rectangle (5-|9) ;
  \fill [red!10,rounded corners]  (5-|2) rectangle (8-|6) ;
  \fill [red!10,rounded corners]  (2-|2) rectangle (5-|3) ;
  \fill [red!10,rounded corners]  (3-|3) rectangle (5-|4) ;
  \fill [red!10,rounded corners]  (4-|4) rectangle (5-|5) ;
  \fill [red!10,rounded corners]  (1-|3) rectangle (2-|6) ;
  \fill [red!10,rounded corners]  (2-|4) rectangle (3-|6) ;
  \fill [red!10,rounded corners]  (3-|5) rectangle (4-|6) ;
  \fill [red!10,rounded corners]  (5-|7) rectangle (6-|9) ;
  \fill [red!10,rounded corners]  (6-|6) rectangle (8-|7) ;
  \fill [red!10,rounded corners]  (7-|7) rectangle (8-|8) ;
  \fill [red!10,rounded corners]  (6-|8) rectangle (7-|9) ;
  \fill [teal!10,rounded corners] (5-|6) rectangle (6-|7) ;
  \fill [teal!10,rounded corners] (5-|1) rectangle (8-|2) ;
  \fill [teal!10,rounded corners] (6-|7) rectangle (7-|8) ;
  \fill [teal!10,rounded corners] (7-|8) rectangle (8-|9) ;
  \end{tikzpicture}
\Body
\substack{-1\\0}    & \substack{0\\1}     & \substack{0\\0}     & \cdots & \substack{0\\0} \;\;   & \;\substack{0\\0}\;    & \cdots        & \substack{0\\0}    \\
\substack{-1\\0}    & \substack{0\\0}     & \substack{0\\1}     & \cdots & \substack{0\\0} \;\;    & \substack{0\\0}    & \cdots        & \substack{0\\0}    \\
\vdots & \vdots  & \vdots  & \ddots & \vdots\;\;  & \vdots & \ddots        & \vdots \\
\substack{-1\\0}    & \substack{0\\0}     & \substack{0\\0}     & \cdots & \substack{0\\1}  \;\;   & \substack{0\\0}    & \cdots        & \substack{0\\0}    \\
-1     & 0       & 0       & \cdots & 0   \;\;    & 1       &        \cdots       &  0      \\
\vdots & \vdots  & \vdots  & \ddots & \vdots \;\; &    \vdots    & \ddots        &    \vdots\\
-1     & 0       & 0       & \cdots & 0     \;\;  &       0 & \cdots               &    1    \\
\vdots & \vdots  & \vdots  & \vdots & \vdots\;\;  & \vdots & \vdots        & \vdots
\end{bNiceMatrix}\]
\phantom{}\\

This matrix is in echelon form with $m$ pivots, meaning it has full rank and is therefore injective. This lets us conclude that $H^1(H_0(\K_-))=0$.
\end{proof}
\begin{note}
    In fact, as we'll see in Theorem 3.2, $H^1(H_0(\K_-))\cong \Z$ whenever $\K$ is neighbourly.
\end{note}

The following lemma was initially proved in \cite[\S3.4]{jonespoly} as Proposition 4 in the context of commuting cubes.

\begin{prop}
    Let $F:2^{[m]}\to \A$. If there's $x\in[m]$ such that for every $J\subseteq [m]\setminus\{x\}$, $F(J\subseteq J\{x\})$ is an isomorphism then $H(F)=0$.
\end{prop}
An extreme case of this is the following corollary, which will be useful in the proof of our main result.
\begin{cor}\label{const}
    For any $X\in\A$, the constant functor $\Delta(X):2^{[m]}\to \A$ is acyclic.
\end{cor}

\section{Result}
   \begin{thm}\label{Truemain}
       Let $\K$ be a simplicial complex on $[m]$. There are maps $\phi_{l,q}:H^l(\H_q(\K_-)) \to H^l(H_q(\K_-))$ for every $q,l\in \Z$ which are isomorphisms whenever $l>2$ or $q>0$, furthermore, we have the following exact sequence 
       \begin{equation}\label{mainseq}
           0\to H^1(H_0(\K_-))\xrightarrow[]{}\Z\to H^2(\H_0(\K_-))\xrightarrow{\phi_{2,0}} H^2(H_0(\K_-))\to 0\tag{\ding{72}}       \end{equation}
       where $\phi_{2,0}$ has a section.
   \end{thm}
   \begin{proof}
       For $q>0$, the result is direct since $\H_q(-)=H_q(-)$, so we only need to show it for $q=0$.\\
       Consider the functor $A:2^{[m]}\to \Z-mod$ given by $A(J)=\left\{\begin{array}{cl}
        \Z & \text{if $J\neq \emptyset$} \\
         0 & \text{else}
    \end{array}\right.$ 
    such that $A(J\xhookrightarrow{}J\cup\{x\})=1$ whenever $J\neq\emptyset$. As $A$ is obtained by taking the constant functor $\Delta(\Z)$ and removing the empty term. By Corollary \ref{const} $H^l(A)=0$ for $l\neq 1$; now, notice that since $\Delta(\Z)$ is acyclic we have that
    \[H^1(A)=\ker d^1_A\cong\ker d_{\Delta(\Z)}^1\cong\text{im }d^0_{\Delta(\Z)}\cong\Z.\]
    
    For every non-empty $J\subseteq[m]$, we have the following short exact sequence
    \[0\to \H_0(\K_J)\to H_0(\K_J)\to \Z\to 0\]

    One can verify this sequence induces a short exact sequence of functors
    \[0\to \H_0(\K_-)\xhookrightarrow{} H_0(\K_-)\to A\to 0,\]
    which in turn, induces a long exact sequence in cohomology. As $A$ is acyclic outside of degree $1$ we have that $H^l(\H_0(\K_-))\cong H^l(H_0(\K_-))$ for $l>2$. Therefore, the only non-trivial part of the sequence is  
    \[0\to H^1(H_0(\K_-))\xrightarrow{} \Z\xrightarrow{} H^2(\H_0(\K_-))\xrightarrow{\phi_{2,0}} H^2(H_0(\K_-))\to 0 \]
    which is the sequence we were looking for. The section of $\phi_{2,0}$ is obtained by assembling the sections of the inclusions $\H_0(\K_J)\xhookrightarrow{}H_0(\K_J)$.   
   \end{proof}
    The previous result presents the difference between double homology and überhomology almost entirely.
   \begin{itemize}
       \item The first part, that is, the isomorphism $H^l(\H_q(\K_J))\xrightarrow{\cong} H^{l}(H_q(\K))$ for $q>0$ or $l>2$ was originaly shown as the main result of \cite{uber}. 
       
       \item For the case $q=-1$, it's straightforward to see that $H_{-1}(\K_-)=0$ and $\H_{-1}(\K_J)=\Z\neq0$ only when $J=\emptyset$, meaning that $H^*(H_{-1}(\K_-))=0$ and $H^*(\H_{-1}(\K
       _-))\cong\Z$ concentrated in degree 0.
       
       \item As for the case $q=0$, Theorem $\ref{Truemain}$ gives us the following two cases:
    \begin{itemize}
        \item [(i)] If $\K$ is neighbourly, then by Lemma \ref{dh} $H^2(\H_0(\K_-))=0$ and by Theorem \ref{Truemain} it follows that $H^1(H_0(\K_-))\cong\Z$ and $H^2(H_0(\K_-))=0$.

        \item [(ii)] If $\K$ isn't neighbourly, then by Lemma \ref{uber} $H^1(H_0(\K_-))=0$ and so we have the short exact sequence 
        \[0\to \Z\to H^2(\H_0(\K_-))\xrightarrow{\phi_{2,0}} H^2(H_0(\K_-))\to 0\]
        
        and since $\phi_{2,0}$ has a section it follows that $H^2(\H_0(\K_-))\cong \Z\oplus H^2(H_0(\K_-))$ 
    \end{itemize}
   \end{itemize}
    
    If we only care about the difference in Betti numbers, we can express the complete difference as the following theorem using Poincaré series.
\begin{defin}\label{Lau}
    Let $\F$ be a field, for a sequence of functors $A=(A_j:2^{[m]}\to \text{vec}(\F))_{j\in \Z}$ we define the \textit{bigraded (Laurent) Poincaré series} $P(A)\in {\Z[[x,x^{-1},y]]}$ where the coefficient of $x^qy^l$ is dim $H^l(A_q)$. 
\end{defin}
\begin{thm}\label{poly}
    Let $\K$ be a simplicial complex on $[m]$ and $\F$ a field, then 
    \[\left(P(\H_*(\K_-;\F))-P(H_*(\K_-;\F))\right)(x,y)=\left\{\begin{array}{cl}
         x^{-1}-y& \text{if $\K$ is neighbourly},  \\
         x^{-1}+y^2&  \text{otherwise.}
    \end{array}\right.\]
\end{thm}
\bibliographystyle{alpha}
\bibliography{ref.bib}
\end{document}